\documentclass[journal,web]{ieeecolor}
\usepackage{generic}
\usepackage{cite}
\usepackage{amsmath,amssymb,amsfonts}
\usepackage{algorithmic}
\usepackage{graphicx}
\usepackage{textcomp}

\usepackage{dsfont}                         
\usepackage{color}
\let\labelindent\relax
\usepackage[shortlabels]{enumitem}

\newcommand{\tens}[1]{%
  \mathbin{\mathop{\otimes}\limits_{#1}}%
}

\newtheorem{rem}{Remark}
\newtheorem{prop}{Proposition}
\newtheorem{lem}{Lemma}
\newtheorem{cor}{Corollary}
\newtheorem{theorem}{Theorem}
\newtheorem{defn}{Definition}

\def\BibTeX{{\rm B\kern-.05em{\sc i\kern-.025em b}\kern-.08em
    T\kern-.1667em\lower.7ex\hbox{E}\kern-.125emX}}
\begin{document}
\title{Safe Dynamic Programming}
\author{Rafal Wisniewski and Manuela L. Bujorianu 
\thanks{The work of the first author has been supported by the Poul Due Jensens Foundation under project SWIft.}
\thanks{R. Wisniewski is with Section of Automation $\&$ Control, Aalborg University, 9220 Aalborg East, Denmark (e-mail: raf@es.aau.dk).}
\thanks{Manuela L. Bujorianu is with Maritime Safety Research Center, Department of Naval Architecture, Ocean $\&$ Marine Engineering, University of Strathclyde, Scotland, UK (e-mail: luminita.bujorianu@strath.ac.uk).}}

\maketitle

\begin{abstract}
We incorporate safety specifications into dynamic programming. Explicitly, we address the minimization problem of a Markov decision process up to a stopping time with safety constraints. To incorporate safety into dynamic programming, we establish a formalism leaning upon the evolution equation. We show how to compute the safety function with the method of dynamic programming. In the last part of the paper, we develop several algorithms for safe dynamic programming.

\end{abstract}

\begin{IEEEkeywords}
Dynamic programming, Markov processes, Optimisation Methods, Safety.
\end{IEEEkeywords}

\section{Introduction}
\label{sec:introduction}
\IEEEPARstart{T}{he} overall objective of the paper is to establish a theoretical background for safe dynamic programming. This calls for  a common technique for dynamic programming and safety. To the best of our knowledge, such a common technique is missing in the literature.  Recently the subject of dynamic programming has enjoyed a resurgence \cite{BertsekasVol1_2017, BertsekasVol2_2018}. The explanation of this increase of popularity is reinforcement learning - a powerful and a prevalent method for learning from data and subsequently generating optimal decisions (see \cite{Babuska10}). In a nutshell, dynamic programming provides the mathematical structure for reinforcement learning. Applications of dynamic programming can be found in robotics \cite{10.1177/0278364913495721}, autonomous vehicles \cite{shalev-shwartz2016safe}, drones \cite{ChangChe-Cheng2020AIoA}, and water networks \cite{jvlIFAC2020}, to name a few examples. In this work, we strive to combine these results with safety \cite{BuWiBou2020}. Safety deals with the subject of assigning the probability of reaching the undesired states - the forbidden set.  The intended result is an optimisation algorithm that keeps the system on the desired safety-level. Specifically, the probability that the process realisations hit the forbidden states before reaching the target set is below a certain value $p$. This is the concept of $p$-safety introduced in \cite{WiBuSL20}.  

We take the starting point of a Markov decision process with a (stationary) policy that for each state provides the probability of choosing a particular control action. Nonetheless, we face a challenge. To compute the optimal path to the target states, we need a random time - the time the process reaches the target set before hitting the forbidden set. Our solution to this challenge is to use the evolution equation~\cite{Helmes01}. It relates the occupation measure with the hitting probability. The occupation measure corresponds to the expected number of the states' visits. When examining the hitting probability, we consider two sets: the set of target states and the set of forbidden states. As the results two functions are derived: the value function $V_\pi$ and the safety function $S_\pi$, both corresponding to a policy $\pi$. The value function is the expected accumulated cost associated to the policy; whereas, the safety function provides the probability of hitting the forbidden set. Interestingly, both functions have a similar form. They comprise the product of the occupation operator (also called the Green operator) and a certain vector. In the case of the value function, this vector is the reward function; and in the case of the safe function, this is the probability of getting to the forbidden set in one step. Hence, not without reason, the occupation operator plays a central role in this work. Indeed, we characterize Markov chains/Markov decision processes using concepts from the probabilistic potential theory, like the infinitesimal generator and its inverse the occupation operator. The use of such concepts leads to elegant analytical proofs of our results.  Importantly, we arrive at a new characterization of the safety function. Specifically, the safety function is the solution of the Bellman's equation and can be computed by an iterative procedure analogous to the one used for computing the value function. Subsequently, we formulate the optimisation with a constraint: the minimisation of the value function $V_{\pi}$ subject to keeping system $p$-safe, $S_\pi \leq p$. As a result, two algorithms for safe dynamic programming are proposed. In the final section, we relax the concept of safety. Subsequently, we develop a local optimisation algorithm; meaning that the control action at each state $i$ is computed only using the information available from its neighbours. By a neighbour of $i$, we understand a state with nonzero transition probability from the state $i$. The salient feature of the Bellman's equation is that it is local.  It transforms the minimization problem over all policies into minimization problems over sets of actions making the space of decision variables much smaller. We introduce a local concept of safety - the relative safety. Equipped with this new concept, we define an optimisation problem. Its solution is local since it is a solution of a certain Bellman's equation.

The subject of minimizing an expected cost without safety constraints is not new; and it is well-known that its solution is obtained by solving the Bellman's equation \cite{BertsekasVol2_2018}.  The problem of safety verification of stochastic systems has also been addressed in the literature~\cite{PrJaPa07}. The paper \cite{BuWiBou2020} has extended the approach based on barrier certificates to discrete settings of Markov chains. The approach of the current paper leans upon elements of probabilistic potential theory. Already in \cite{CDC17} and \cite{CDC19}, it has been shown that the analytical approach based on potential theory provides straightforward proofs for the barrier certificate's properties. The problem of constrained optimisation of a Markov decision process has been addressed before in \cite{Altman99constrainedmarkov}. The proposed solution is a linear programming formulation. In fact, we have used the technique of \cite{Altman99constrainedmarkov} in the proof of Theorem~\ref{Thm.ConstraintOptimisation} showing that minimizing the expected cost subject to the safety constraint has no duality gap. Its trace is also visible in the algorithm formulated in Subsection~\ref{Subsection.LP}. A pragmatic methods for safe dynamic programming and reinforcement learning have been addressed in \cite{Marvi2020}. In the above reference, safety is ensured with a barrier function, which serves as a soft constraint to the system.  On the other hand, the work \cite{li2018safe} proposes a supervisor that prevents the applied control action to drive the system into unsafe regions. 

The list of original contributions of this work includes:
\begin{enumerate}
    \item $p$-safety is re-formulated as a dynamic programming problem.
    \item The occupation operator is introduced into dynamic programming.
    \item The safe dynamic programming is formulated as optimisation with constraints, and we prove that the problem has no duality gap. Therefore, it can be solved in terms of dual optimisation problem.
    \item Two algorithm for solving safe dynamic programming are proposed.
    \item A local concept of $q$-relative safety is introduced.
    \item A local algorithm for (relatively) safe dynamic programming is proposed.
\end{enumerate}

The paper is organised as follows. After introducing the notations in Sections~\ref{Sec.Notations}, we shed light on the preliminary objects of this work: the occupation operator, the occupation measure, the hitting probability, and the evolution equation in Section~\ref{Sec.MarkovChains} and \ref{Sec.EvolutionEquation}. The model of the controlled process adopted in this paper is the Markov decision process. Its description is summarized in Section~\ref{Sec.MDP}. The dynamic programming with stopping time is developed in Section~\ref{Sec.DynamicProgramingStoppingTime}. It is shown that the safety function can be computed as the solution of the Bellman's equation in Section~\ref{Section.Safety}. The main results about safe dynamic programming are formulated in Sections~\ref{Sec.safeDynamicProgramming}. Subsequently, several algorithms for safe dynamic programming are developed in Sections~\ref{Section.SafeDynamicProgramming} and \ref{Sec.LocalSafeDynamicProg}. An illustrative example is provided in Section~\ref{Sec.Example}.

\section{Notation}\label{Sec.Notations}

For a finite set $U$, we write $R^{U} := R^{|U|}$. We use $I_U$ to denote the identity matrix on $U$, and $\mathds{1}_U$ to denote the vector of ones on $U$.
 
Occasionally, we will make use of multi-linear operators (multidimensional matrices) of the form $A = (a_{\alpha_1 \alpha_2 \hdots \alpha_m}) \in \mathds{R}^{n_1} \times \mathds{R}^{n_2} \times \hdots \times \mathds{R}^{n_m}$.  We define $(i,j)$-product of a multidimensional matrix $A$ and $B$ by
\begin{align*}
    A \tens{(i,j)} B = \sum_{\gamma = 1}^{n_i} (a_{\alpha_1 \hdots \alpha_{i-1} \gamma \alpha_{i+1} \hdots  \alpha_m} b_{\beta_1 \hdots \beta_{j-1} \gamma \beta_{j+1} \hdots  \beta_l}). 
\end{align*}
Explicitly, if $A$ and $B$ are standard matrices 
\begin{align*}
AB = A \tens{(2,1)} B.
\end{align*}
We introduce a diagonal operator 
\begin{align*}
\mathrm{diag}_i(A)= (c_{\alpha_1  \hdots \alpha_i \hdots  \alpha_{m-1}}), 
\end{align*}
where $c_{\alpha_1 \hdots \alpha_i \hdots  \alpha_{m-1}} = a_{\alpha_1 \hdots \alpha_{i-1} \alpha_i \alpha_{i} \alpha_{i+2} \hdots  \alpha_m}$.
Specifically, for a standard matrix $A$, $\mathrm{diag}_1(A)$ is the vector consisting of diagonal entries $a_{ii}$ of $A$.

For two vectors $v, w \in \mathds{R}^m$, we will use the Hadamard product $v \circ w$ defined by
\begin{align*}
    (v \circ w)(i) = v(i) w(i). 
\end{align*}
The notation $v \geq 0$ denotes $v(i) \geq 0$ for all $i \in \{1,..., m\}$.

\section{Markov Chains}\label{Sec.MarkovChains}

Let $\mathcal{X}$ be a countable set of states. We denote the states in $\mathcal{X}$ by the letters $i,j$.
The $\sigma $-algebra on $\mathcal{X}$ is the algebra of all its subsets. 
  A probability distribution $\nu$ on $\mathcal{X}$ is a sequence $(\nu (j))_{j\in \mathcal{X}}$ thought of as a row vector $\nu \in \mathds{R}_{\geq 0}^{\mathcal{X}}$. A function $f:\mathcal{X}\rightarrow \mathds{R}$ 
is defined as a column vector $f=(f(j))_{j\in \mathcal{X}}^{\top}$. 

We suppose that $(X_{t}) := (X_{t})_{t \in \mathds{N}}$ is a discrete-time homogeneous Markov chain with the
transition probabilities 
\begin{align}\label{EntriesOfTransitionMtx}
p_{ij} := \mathds{P}[X_{t+1}=j|X_{t} =i]=\mathds{P}[X_1=j|X_{0} =i].
\end{align}
The transition matrix $P$ of $(X_t)$ is $P:=(p_{ij})_{i,j\in \mathcal{X}}$. 
The $k$-step transition probabilities are $\mathds{P}[X_k=j|X_0 =i]= (P^k)_{ij}$, where $P^k =P P ...P$ is the $k$-fold matrix product.

The transition matrix $P$ acts on the right on functions, $f \mapsto Pf$, and on the left on the measures $\nu \mapsto \nu P$.

Let $H$ be an arbitrary subset of $\mathcal{X}$, which will be kept fixed. We restrict the transition probabilities of the Markov chain  $(X_{t})$ to the set $H$.
These are the taboo transition probabilities \cite{Syski73}. We collect the taboo transition probabilities into the transition matrix $Q = (p_{ij})_{i,j \in H}$. In this case, the transition matrix $Q$ is substochastic, i.e., the sum of row entries $\sum_{j \in H} p_{ij} \leq 1$.

We introduce the occupation operator $G$. It will be of paramount importance to the entire article. This is defined as follows:
\begin{align}\label{Green_op}
    G := \sum ^{\infty}_{k=0}Q^k.
\end{align}
The occupation operator is also called the Green operator of the transition matrix $Q$. 
We can write the entries of $G$ as 
\begin{align}\label{Green_fc}
    G(i,j)=\sum^{\infty}_{t=0}\mathds{P}[X_t =j| X_0 = i, ~ X_0, \hdots, X_{t-1} \in H] 
\end{align}
for all $i,j \in H$.
Intuitively, $G(i,j)$ represents the number of visits of the state $j$ starting from $i$ while evolving in the taboo set $H$. Hence, the name `occupation'.
From \eqref{Green_op}, it follows that 
\begin{align} \label{Eq.RelationGreenGenerator}
    G = I_H + QG = I_H + GQ;
\end{align}
recall $I_H$ is the identity matrix on $H$.
For a function $f$ on $H$, $Gf$ is called the potential of the function $f$. $Gf$ is to be understood as the multiplication of the matrix $G$ by the vector representation of $f$.

\section{Evolution of the Markov Chain}
\label{Sec.EvolutionEquation}
To study the reach-avoidance problem (reach the target set while avoiding the forbidden set), we examine the process up to the first hitting time of a target set or a forbidden set. We associate a reward to each state and ask two questions: What is the cost of getting to the target set, and what is the probability that the process reaches the forbidden set before the target set? To this end, we will use the evolution equation relating the occupation measure and the hitting probability, which we characterize first.

\subsection{Occupation measure}

Suppose that $\tau$ is a stopping time and $D$ is a subset of $\mathcal{X}$. Let $\rho_{<\tau}(D)$ be 
a random variable that describes the amount of time the Markov chain spends in $D$ before time $\tau$ has passed. Formally, the occupation variable $\rho_{<\tau}(D)$ is written as 
\begin{align} \label{occ_meas_rv}
    \rho_{<\tau}(D) := \sum_{t=0}^{\tau-1} I_{\{X_t \in D\}}.
\end{align}
The occupation measure $\gamma_{<\tau}$ for $(X_n)$  is defined as the expectation of $\rho_{<\tau}(\cdot)$ in \eqref{occ_meas_rv}, i.e., $\gamma_{<\tau}(D) := \mathds{E}\rho_{<\tau}(D)$. 
From the calculation
\begin{align}\label{occ_meas_est}
    \gamma_{<\tau}(D)=\sum_{t=0}^{\infty} \mathds{P}[t < \tau|X_t \in D],
\end{align}
it follows that $\gamma_{<\tau}$ is indeed a measure on all subsets of $\mathcal{X}$, as its name suggests.
 If $\mu$ is the initial measure ($i$ is the initial state), then we employ the probability $\mathds{P}^{\mu}$ ($\mathds{P}^{i}$), and use
 the notation $\gamma_{<\tau}^{\mu}$ ($\gamma_{<\tau}^{i}$), i.e.,
 \begin{align*}
  \gamma_{<\tau}^{\mu}(D)=\sum_{t=0}^{\infty} \mathds{P}^{\mu}[t < \tau|X_t \in D].
 \end{align*}

We define the integral with respect to $\gamma_{<\tau}$ of a  vector function $f$ as
\begin{align}\label{Eq:IntegrationOccupation}
    \langle \gamma_{<\tau}, f \rangle := \mathds{E}\sum_{t=0}^{\tau-1}f(X_t).
\end{align}

The equation~\eqref{Eq:IntegrationOccupation} will be instrumental for computing the accumulated cost of the process until stopping time $\tau$. Computing the minimal value of $\langle \gamma_{<\tau}, f \rangle$ for different policies is the object of interest of dynamic programming.

\subsection{Hitting probabilities}
For a stopping time $\tau$, let $\lambda_\tau(D)$ be the expected time that the process lies in a set $D \subset \mathcal{X}$ at the time $\tau$,
\begin{align}\label{hit_prob}
    \lambda_\tau(D) := \mathds{P}[\tau< \infty |X_\tau \in D ] =\sum_{t=0}^{\infty} \mathds{P}[\tau=t|~X_t \in D].
\end{align}
When the initial state is $i$, we employ the probability $\mathds{P}^{i}$, and use the notation $\lambda_\tau ^{i}$. Similarly, for the initial probability $\mu$, we use $\mathds{P}^{\mu}$, and the notation $\lambda_\tau ^{\mu}$ for the hitting distribution.

We define the hitting operator corresponding to the stopping time $\tau$ as the integral of a measurable function $f$ with respect to $\lambda_\tau$ as
\begin{align}\label{hitting_op}
   \langle \lambda_\tau , f \rangle = \mathds{E}(f(X_\tau) I_{[\tau<\infty]}).
\end{align}


Specifically, let $E$ and $U$ be disjoint subsets of $X$. We think about $E$ as a target set and $U$ as a forbidden set. Suppose that $\tau = \tau_{U \cup E}$, the first hitting time of the union of $U$ and $E$. Then
\begin{align*}
     \langle \lambda_{\tau_{U \cup E}} , I_U \rangle
\end{align*}
is the probability that that the process hits the forbidden set before the target set. This relation will be instrumental for the computation of safety.

\subsection{Evolution Equation}

Let $\mu$ be an initial distribution on $\mathcal{X}$. For a stopping time $\tau$ of the Markov chain $(X_n)$, as in the previous chapter, let $\gamma_{<\tau}$ denote the occupation measure, and $\lambda_\tau$ the hitting probability.
The connection between the occupation measure and the hitting probability is known in the literature as the adjoint or evolution equation \cite{Helmes01},
\begin{align} \label{ev_eq}
    \lambda_\tau^{\mu} = \mu +\gamma_{<\tau}^{\mu} \mathcal{L},
\end{align}
where $\mathcal{L}$ is the generator,   $\mathcal{L} = P - I$.
 The triplet $(\mu,\gamma_{<\tau}^{\mu}, \lambda_\tau^{\mu})$ characterises in a unique way the underlying Markov process~\cite{Bratt93}.

\section{Markov Decision Processes}\label{Sec.MDP}

We generalise \eqref{EntriesOfTransitionMtx} to Markov Decision Processes. To this end, we suppose that $(U_n)$ is a process with values in a countable set $\mathcal{U}$ of actions, and study the conditional probabilities $\mathds{P}[X_{t+1}=j|~X_{t} =i,~U_{t} = u]$ for $i, j \in \mathcal{X}$, and $u \in \mathcal{U}$.
We remark that Markov property holds for the MDP,
\begin{align}\label{Markov_MDP}
   & \mathds{P}[X_{t+1}=j|~X_0 = i_0,~U_0 = u_0, \hdots  ,X_{t} =i_t,~U_{t} = u_t]  \nonumber \\
    = & \mathds{P}[X_{t+1}=j|~X_{t} =i_t,~U_{t} = u_t],
\end{align}
and introduce transition probabilities
\begin{align*}
	p_{i u j}  = \mathds{P}[X_{t+1}=j|X_{t} =i,~U_{t} = u],
\end{align*}
where $(i,u,j) \in \mathcal{X} \times \mathcal{U} \times \mathcal{X}$. We think about the multidimensional matrix 
$$P_{\mathrm{c}} := (p_{iuj}) := (p_{iuj})_{(i,u,j) \in \mathcal{X} \times \mathcal{U} \times \mathcal{X}}$$ 
as a priori model, where for a given control value $u$, the transition probability from the state $i$ to $j$ is specified. 

By a stochastic policy, we understand the family of stochastic kernels $(\pi_{iu}(t))_{(i,u) \in \mathcal{X} \times \mathcal{U}}$,
\begin{align*}
\pi_{iu} (t) = \mathds{P}[U_t=u|X_t =i].
\end{align*}
We think about the policy $\pi$ as the to-be-designed stochastic control.
In this work, we entirely restrict our attention to stationary policies; the stochastic policy is stationary if $\pi_{iu}$ does not depend on the time, i.e., 
\begin{align*}
\pi_{iu}  = \mathds{P}[U_t=u|X_t =i] = \mathds{P}[U_0=u|X_0 =i].
\end{align*}
To conclude, the stationary policy $\pi$ is seen as the (possibly infinite dimensional) matrix  
\begin{align}\label{Eq.PolicyMatrix}
\pi  := (\pi_{iu})_{(i,u) \in \mathcal{X} \times \mathcal U}    
\end{align}
with entries between $0$ and $1$, corresponding to the probability that at the state $i$, the control action has the value $u$.

Let $\mathcal{U}$ be an ordered set of actions $u$. On the Euclidean space $R^{\mathcal{U}}$, we consider a basis - the unit vectors $(e^u)_{u \in \mathcal{U}}$; $e^u$ has its $u$th coordinate equal to $1$ and the other coordinates equal to $0$. Let $\mathcal{D} $ be the standard simplex in $\mathds{R}^{\mathcal U}$, 
\begin{align}
    \mathcal{D}:=\{\alpha = (\alpha_{u})_{u \in \mathcal{U}}|\alpha_u \geq 0\text{, }\sum_{u \in \mathcal{U}} \alpha_u = 1\}.
\end{align}
Consequently, for each fixed $i$, the transition probabilities  $(\pi_{iu})_{u \in \mathcal{U}} \in \mathcal{D}$.
In the following, we slightly abuse the notation concerning the policy $\pi$. We write $\pi \in \mathcal{X} \times \mathcal{D} = \mathcal{D}^{\mathcal{X}}$ even though  we mean its matrix representation $(\pi_{iu})_{(i,u) \in \mathcal{X} \times \mathcal{U} }$ with $\sum_{u \in \mathcal{U}} \pi_{iu} = 1$ for each $i \in \mathcal{X}$. 
 

For a given policy $\pi$, we use the law of total probabilities and compute the transition probability of the induced Markov chain, 
\begin{align}\label{total_TP}
	p_{ij}(\pi) = \sum_{u \in \mathcal{U}} \pi_{iu} p_{i u j}.
\end{align}
For the policy $\pi$, we define the transition probability matrix 
$$P(\pi) = (p_{ij}(\pi))_{(i,j) \in \mathcal{X} \times \mathcal{X}}.$$

Using the $(i,j)$-product and $\mathrm{diag}_i$ operators, the transition probability matrix $P(\pi)$ is  
\begin{align}\label{Eq.PisAffine}
    P(\pi) = \mathrm{diag}_1 (\pi \tens{(2,3)} P_{\mathrm{c}}).
\end{align}
We conclude that $P(\cdot)$ is a linear map. 

\begin{rem}
Our notation is chosen such that the families are stochastic with respect to the last subscript, i.e., 
\begin{align*}
	\sum_{u \in \mathcal{U}} \pi_{iu} = 1 \hbox{ and }  \sum_{j \in \mathcal{X}} p_{iuj} = 1. 
\end{align*}
\end{rem}
~\\~

We finish this section by extending the Green operator to Markov decision processes. For a policy $\pi$, we define
\begin{align}\label{Eq.DefGreen}
       G(\pi) := \sum ^{\infty}_{k=0}Q(\pi)^k, 
\end{align}
where
$Q(\pi) = (p_{ij}(\pi))_{(i,j) \in H \times H}$. Recall, that $H$ is a taboo set, a subset of $\mathcal{X}$, and that we have restricted the Markov chain  $(X_{t})$ to $H$.

\section{Dynamic Programming with Stopping Time}\label{Sec.DynamicProgramingStoppingTime}


At the outset, we introduce the notion of a reward 
$$\rho:  \mathcal{U} \times \mathcal{X}\to \mathds{R}.$$
The function $\rho$ induces a process $(\rho_t)$ by
\begin{align} \label{Eq.RewardProcess}
	\rho_t = \rho(U_t,X_t).
\end{align}
We suppose that the process $(U_t)$ is generated by a policy $\pi$. Let $\tau$ be a stopping time, for example the first hitting time of a set. The cost for the policy $\pi$ until time $\tau$ is
\begin{align}\label{Eq.OptimisationFunction}
    V_{\pi}(i) := \mathds E_{\pi} \left[\sum_{t=0}^{\tau} \rho_t| X_0 = i\right],
\end{align}
and the vector $V_{\pi}=(V_{\pi}(i))_{i \in \mathcal{X}}$. 

The aim of this section is to evaluate the cost function $V_{\pi}$, and subsequently to find a minimizing static policy. All the objects used below will depend on the policy $\pi$; herein, the expectation, the transition matrix, occupation measure, and hitting probability. Therefore to enhance readability, we occasionally suppress $\pi$ from the notation. 

At the outset, we notice that in \eqref{Eq.RewardProcess}, $\tau$ is a random variable; therefore, the expectation operator cannot be moved under the summation symbol, as it is customarily done in standard dynamic programming/reinforcement learning (see \cite{BertsekasVol2_2018} and \cite{Babuska10}). Instead, to evaluate the cost function $V_{\pi}$ in \eqref{Eq.OptimisationFunction}, we employ the evolution equation introduced in Section \ref{Sec.EvolutionEquation},
\begin{align} \label{ev_eq_function}
    \langle \lambda_\tau^{\mu}, f \rangle = \langle \mu, f \rangle + \langle \gamma_{<\tau}^{\mu}, \mathcal{L} f \rangle.
\end{align}
This technique will be visible in the proof of the next proposition. In accordance with our convention, all the measures and the generator $\mathcal{L}$ in \eqref{ev_eq_function} depend on the policy $\pi$. 
~\\~

\begin{prop}\label{RelationGreenAndValue}
For a Markov decision process $(X_t)$ on the set $\mathcal{X}$ of states and the action set $\mathcal{U}$, let $\pi$ be a fixed policy and $\rho$ be the reward. Let $E \subset \mathcal{X}$ be a target set, and  $H := \mathcal{X} \setminus E$. Suppose that
\begin{enumerate}
    \item the decision process $X_t$ with policy $\pi$ is transient on $H$, 
    $$\mathds{P}_{\pi}\left[X_{t} \in H \text { for infinitely many } t \mid X_{0} \in H \right]=0;$$
    \item $\rho(u, e) = 0$ for all $u \in \mathcal{U}$ and all $e \in E$.
\end{enumerate}
Suppose that $\tau := \tau_{E}$ is the first hitting time of the target set $E$, i.e., $\tau_{E} := \min\{t \geq 0|~X_t \in E\}$.

Let $ R_{\pi} := \mathrm{diag}_1( \pi \rho)|_H,$ where $\pi$ is the policy matrix in \eqref{Eq.PolicyMatrix}, and $\rho = (\rho(u,i))_{(u,i) \in \mathcal{U} \times \mathcal{X}}$; or in components $R_{\pi}(i) = \sum_{u\in\mathcal{U}} \pi_{iu} \rho(u,i) $.

Then the cost function $V_{\pi}$ in \eqref{Eq.OptimisationFunction} restricted to $H$ is 
\begin{align*}
    V_{\pi}|_{H} = G(\pi) R_{\pi},
\end{align*}
where $G(\pi)$ is the Green kernel in \eqref{Eq.DefGreen} for the set $H$, and $V_{\pi}(e) = 0$ for $e\in E$.
\end{prop}
~\\~

\begin{proof}
We write $\mathds{P} := \mathds{P}_{\pi}$, and $\mathds{E} := \mathds{E}_{\pi}$.
We define the following random variable 
\begin{align*}
    r_{\pi}(X_t):= \mathds{E} [\rho(U_t,X_t) | X_t] = \sum_{u \in \mathcal U}\rho(u,X_t) \mathds{P}[U_t = u|X_t]
\end{align*}
We claim that
\begin{align*}
   V_{\pi}(i) & = \mathds{E}\left [ \sum_{t=0}^{\tau} \rho_t|X_0 = i \right ] =  \mathds{E}\left [ \sum_{t=0}^{\tau} r_{\pi}(X_t)|X_0 = i \right ].
\end{align*}
The claim follows from
\begin{align}
     V_{\pi}(i)  & = \sum_{k = 0}^{\infty} \mathds{E} \left[ \sum_{t=0}^{k} \rho_t|X_0 = i \right] \mathds{P}[\tau = k|X_0 = i] \nonumber \\
     & = \sum_{k = 0}^{\infty} \sum_{t=0}^{k} \mathds{E} \left[\rho_t|X_0 = i \right] \mathds{P}[\tau = k|X_0 = i]. \label{Eq.HowToRepresentCost}
 \end{align}    
We observe that
\begin{align*}
    \mathds{E} \left[\rho_t|X_0 = i \right] = \sum_{j\in \mathcal{X}}\mathds{E} [\rho_t|X_t = j] \mathds{P}[X_t=j| X_0 = i],
\end{align*}
and
\begin{align*}
     \mathds{E}[\rho_t|X_t = j] 
     & = \sum_{u \in \mathcal{U}}\rho(u,j) \mathds{P}[U_t = u|X_t = j] \\
     & = \mathds{E}[r_{\pi}(X_t)|X_t = j] =  r_{\pi}(j). 
\end{align*}
Hence, the claim follows from substituting $\mathds{E}_{\pi}[\rho_t|X_t = j]$ by $r_{\pi}(j)$ in \eqref{Eq.HowToRepresentCost}.
Since $R_{\pi}(i) = r_{\pi}(i)$ for $i \in E$, from \eqref{Eq:IntegrationOccupation}, we conclude that 
\begin{align*}
       V_{\pi}(i) & = \mathds{E}\left [ \sum_{t=0}^{\tau} r_{\pi}(X_t)|X_0 = i \right ] \\
       & = \langle \gamma_{< \tau}, R_{\pi} \rangle + \mathds{E}_{\pi}[r_{\pi}(X_\tau)|X_0 = i].
\end{align*}
As $\rho(u, e) = 0$ for all $u \in \mathcal{U}$ and all $e \in E$, we have
\begin{align*}
       V_{\pi}(i) = \langle \gamma_{< \tau}^i, R_{\pi} \rangle.
\end{align*}

In the second part of the proof, we will use the evolution equation \eqref{ev_eq}, to evaluate $\langle \gamma_{< \tau}^{\mu}, R(\pi) \rangle$.
We claim that
\begin{align*}
    0 = \langle \mu|_H, G(\pi) f|_H \rangle - \langle \gamma_{<\tau}^{\mu},  f|_H \rangle,
\end{align*}
for any $f$ such that $f(e) = 0$ for all $e \in E$, and for any initial measure $\mu$.
The claim leads us to the conclusion
\begin{align*}
    0 = G(\pi) R_{\pi} - V_{\pi}|_{H}.
\end{align*}

To prove the claim, without loss of the generality, we suppose that the states are numbered such that the first states belong to $H$ and the remaining to $E$. We decompose the transition matrix $P := P(\pi)$ (possibly infinite dimensional) as follows
\begin{align*}
    P = \begin{bmatrix} Q & P_E^H \\ P_H^E & P_E^E\end{bmatrix}.
\end{align*}
Consequently,
\begin{align*}
    \mathcal{L} = \begin{bmatrix} Q - I_H & P_E^H \\ P_H^E & P_E^E - I_E\end{bmatrix},
\end{align*}
where $I_H$ is the identity matrix on $H$, similarly $I_E$ is the identity matrix on $E$, and $Q := P(\pi)|_H$, as above.
We define a matrix
$$
\tilde{G}=\left[\begin{array}{ll}G & 0 \\ 0 & 0\end{array}\right],
$$
where $G$ is the Green operator defined in \eqref{Green_op}.
By the relation \eqref{Eq.RelationGreenGenerator},
$$\mathcal{L} \tilde{G}= - \left[\begin{array}{cc}I_H & 0 \\ 0 & 0\end{array}\right],$$
and
\begin{align*} 
    \lambda_\tau^{\mu}|_H G = \mu|_H G - \gamma_{<\tau}^{\mu}|_H,
\end{align*}
On the other hand, $\tau$ is the first hitting time of $E$, therefore $\lambda_\tau^{\mu}|_H = 0$. In conclusion,
\begin{align*} 
    0 = \mu|_H G - \gamma_{<\tau}^{\mu}|_H.
\end{align*}
Suppose that $f = \begin{bmatrix} f|_H & f|_E\end{bmatrix}^{\mathrm{T}} = \begin{bmatrix} f|_H & 0 \end{bmatrix}^{\mathrm{T}}$ in \eqref{ev_eq_function}, 
\begin{align*}
    0 = \langle \mu|_H, G f|_H \rangle - \langle \gamma_{<T}^{\mu},  f|_H \rangle.
\end{align*}
This proves the claim as $f|_H$ is arbitrary. We conclude that $V_{\pi}|_{H} = G(\pi) R_{\pi}$. 
\end{proof}
~\\~


As in Section~\ref{Sec.MDP}, we characterize the set of policies 
$$\pi = (\pi_{iu})_{(i,u) \in H \times \mathcal{U}} \in \mathcal{D}^H := \underbrace{
\mathcal{D} \times \ldots \times \mathcal{D}}_{|H| \hbox{ times}},$$ 
since $\sum_{u\in \mathcal{U}} \pi_{(i,u)} = 1$ for all $i \in H$.

We strive to solve the following optimisation problem 
\begin{align}\label{Eq.ValueFunctionCompact}
   V^*(i) = \min_{\pi \in \mathcal{D}^{H}} V_\pi(i).
\end{align}
We call $V^*$ the value function.

We observe that $G(\pi) = I_H + G(\pi)Q(\pi)$,  hence 
\begin{align*}
    V_{\pi}|_{H} = R_{\pi} + Q(\pi)G(\pi)R_{\pi},
\end{align*}
and the result is the celebrated formula known in dynamic programming
\begin{align}\label{Eq.BellmanPolicy}
    V_{\pi}|_{H}  = R_{\pi} + Q(\pi)V_{\pi}|_H.
\end{align}
with the boundary condition $V_{\pi}(e) = 0$ for $e \in E$. 

We introduce the discrete Laplacian operator $\Delta(\pi)=-\mathcal{L}(\pi) = I - Q(\pi)$, and write \eqref{Eq.BellmanPolicy} as
\begin{align*}
    \Delta(\pi) V_{\pi}|_{H}  = R_{\pi}.
\end{align*}
Concretely, if the reward function $\rho$ is non-negative then $V_{\pi}|_H$ is a superharmonic function.

The solution of \eqref{Eq.BellmanPolicy} is unique since for another solution $W: \mathcal{X} \to \mathds{R}$, 
\begin{align}\label{Eq.ResultDirichletProblem}
   \Delta(\pi) (V_{\pi} -W)=0 \hbox{ on } H,
\end{align}
where $\Delta(\pi)=I - Q(\pi)$ is the discrete Laplacian operator, with the boundary condition $V_{\pi}(\tau)-W(\tau)=0$. We conclude that \eqref{Eq.ResultDirichletProblem} is the Dirichlet problem which has the unique solution $V_{\pi} -W = 0$.

The function $V_{\pi}$ can be computed as the limit of the sequence $(V_{\pi}^{n})$ defined by the iteration
\begin{align} \label{Eq.IterativeValue}
    V_{\pi}^{n+1} =  R_{\pi} + Q(\pi)V_{\pi}^{n},
\end{align}
with an arbitrary initial condition $V_{\pi}^{0}$. Unfolding~\eqref{Eq.IterativeValue} gives 
\begin{align}
    V_{\pi}^{n} =  \sum_{k=0}^n Q^k(\pi) R_{\pi} + Q(\pi)^n V_{\pi}^{0}. 
\end{align}
The sequence  $(V_{\pi}^{n})_{n \in \mathds{N}}$ converges to
\begin{align*}
    V_{\pi}^{\infty} = G(\pi) R_{\pi}.
\end{align*}
as $Q(\pi)$ is the restriction of the transition matrix on $H$ corresponding to the transient chain; thereby, $Q(\pi)^n \to 0$.
In conclusion,  $V_{\pi}^{\infty} =  V_{\pi}|_{H}.$

We recall here the standard results of reinforcement learning \cite{BertsekasVol2_2018}, which will be instrumental in the following sections. The proof sheds light on the usefulness of the Green operator.  
~\\~
\begin{lem}\label{Lemma.RLAlgorithms}
Suppose that there is at least one policy $\pi$ that the premises of Lemma~\ref{RelationGreenAndValue} are satisfied. 
The optimal cost $V^*$ restricted to the set $H$ satisfies Bellman's equation
\begin{equation} \label{eq.BellmanEquationcompact}
    V = \min_{\pi \in \mathcal{D}^{H}} \left[ R_{\pi} + Q(\pi)V \right].
\end{equation}

Furthermore, the function $V^{*}|_{H}$ is the (coordinate-wise) limit of the sequence $(V^n)$ defined by
\begin{align}\label{Eq.IterativeBellaman1}
    V^{n+1} = \min_{\pi \in \mathcal{D}^{H}} \left[ R_{\pi} + Q(\pi)V^n \right]
\end{align}
with an arbitrary initial condition $V^0 \geq 0$. 
\end{lem}
~\\~

\begin{proof}
By hypothesis $V^*$ is finite. We frequently and without mentioning use the observation that if $V \geq 0$ component-wise, and so is $Q(\pi)V \geq 0$.

Firstly, we will show that if $V$ satisfies \eqref{eq.BellmanEquationcompact} then $V = V^*|_H$. To this aim, it is enough to show that $V \leq V^*|_H$. For any ${\tilde \pi \in \mathcal{D}^{H}}$, we have 
\begin{align*}
    V &= \min _{\pi \in \mathcal{D}^{H}} \left[ R(\pi) + Q(\pi)V \right] \leq R(\tilde \pi) + Q(\tilde \pi)V \\
    & =  R(\tilde \pi) + Q(\tilde \pi) \left(
    \min _{\pi \in \mathcal{D}^{H}} \left[ R(\pi) + Q(\pi)V \right] \right)\\
    & \leq R(\tilde \pi) + Q(\tilde \pi) \left( R(\tilde \pi) + Q(\tilde \pi)V \right) \leq G(\tilde \pi) R( \tilde \pi) \\ &= V_{\tilde \pi}|_{H}.
\end{align*}
In the above, the components of $V_{\tilde \pi}$ might be infinite.

We will show that the sequence $(V^n)$ converges to $V$. To this end, we construct two sequences $(\pi^n)$ and $(W^n)$ by 
\begin{align*}
    \pi^n \in \arg \min_{\pi \in \mathcal{D}^{H}} \left[ R(\pi) + Q(\pi)V^n \right]
\end{align*}
and $W^n = G(\pi^n) R(\pi^n)$. For each $n$, a sequence $(W_m^n)_{m \in \mathds{N}}$ defined by
\begin{align*}
    W^n_{m+1} := R(\pi^n) + Q(\pi^n) W^n_{m} 
\end{align*}
with $W^n_0 = V^n$ converges to $W^n$.

We define an operator $T$ acting on the functions on $H$ by
\begin{align*}
    T: f  \mapsto \min_{\pi \in \mathcal{D}^{H}} \left[ R(\pi) + Q(\pi)f \right]
\end{align*}
for $f \in \mathds{R}^H$. 
We conclude that
\begin{align*}
    W^n_{m} \geq T^m V^n = V^{m+n},
\end{align*}
and $W^n \geq W^{n+1}$. Furthermore,
\begin{align*}
    \lim_{n \to \infty} W^n = V.
\end{align*}
By contradiction argument, we also conclude that
\begin{align*}
    \lim_{n\to\infty} T^n V^0 = V.
\end{align*}
\end{proof}

We will shed light on \eqref{Eq.IterativeBellaman1}. To this end, we use the following operator
\begin{align*}
    T: V  \mapsto \min_{\pi \in \mathcal{D}^{H}} \left[ R(\pi) + Q(\pi)V \right],
\end{align*}
and study the components $TV(i)$ for fixed $i \in H$,
\begin{align*}
    R_{\pi}(i) = \sum_{u \in \mathcal{U}} \pi_{iu} \rho(u,i),
\end{align*}
and 
\begin{align*}
    (Q(\pi)V)(i) = \sum_{j \in H} \sum_{u \in \mathcal{U}} \pi_{iu} p_{i u j} V(j).
\end{align*}
We observe that $R_{\pi}(i)$ and  $(Q(\pi)V)(i)$ depend only on the distribution of $\pi_{i} \in \mathcal{D}$. Hence, we write
\begin{align*}
TV(i) = \min_{\pi_{i} \in \mathcal{D}}   R(\pi)(i)  + (Q(\pi)V)(i),
\end{align*}
and conclude that optimisation at the state $i$ is local - it involves only the local information of the probability distribution of the policy at $i$.  

In the last part of this section, we follow the idea of \cite{Altman99constrainedmarkov}, and show that the value function \eqref{Eq.ValueFunctionCompact} is the largest among the functions  $V: E \to \mathds{R}$ satisfying the inequality
\begin{align*}
    \Delta(\pi) V \leq R_{\pi} \hbox{ for all } \pi \in \mathcal{D}^{H}.
\end{align*}

~\\~
\begin{lem}\label{SupremumValueunction}
The value function satisfies
\begin{align}\label{Eq.ValueSupremumFormula0}
    V^*|_H= \sup \mathcal{V},
\end{align}
where
\begin{align*}
    \mathcal{V} := \{V \in \mathds{R}^H |~\Delta(\pi) V \leq R_{\pi} \hbox{ for all } \pi \in \mathcal{D}^{H}\},
\end{align*}
and $\sup$ is to be understood coordinate-wise.
\end{lem}
~\\~

\begin{proof}
Suppose that $V \in \mathcal{V}$. Then for all $\pi  \in \mathcal{D}^{H}$, we have
\begin{align*}
    \Delta(\pi) V + a_{\pi} = R_{\pi},
\end{align*}
where $a_{\pi} \geq 0$. We write
\begin{align*}
    \Delta(\pi) ( V + G(\pi) a_{\pi}) = R_{\pi}.
\end{align*}
Therefore, there is $V_{\pi}$ such that $\Delta(\pi) V_{\pi} = R_{\pi}$. Since
\begin{align*}
    V_{\pi} = V + G(\pi) a_{\pi},
\end{align*}
and $G(\pi)$ is a positive matrix (it has non-negative entries), we have 
\begin{align*}
    V_{\pi} \geq V.
\end{align*}
Since the above equality holds for an arbitrary policy, it follows that 
\begin{align*}
    V^*|_H \geq V.
\end{align*}
for any $V$ in $\mathcal{V}$. 
On the other hand, the value function $V^*$ satisfies the Bellman's equation
\begin{align*}
    \min_{\pi \in \mathcal{D}^H} \Delta(\pi) V^*|_H = R_{\pi}.
\end{align*}
Hence,
\begin{align*}
    \Delta(\pi) V^*|_H \geq R_{\pi} \hbox{ for all } \pi \in \mathcal{D}^{H}.
\end{align*}
We conclude that
 \begin{align*}
     V^*|_H \in \mathcal{V}.
 \end{align*}
Thereby, \eqref{Eq.ValueSupremumFormula0} follows. 
\end{proof}

\section{Safety}\label{Section.Safety}
In this section, we formulate safety as a dynamic programming problem. 
In the previous section, we have considered the terminal set $E$ and its complement, the taboo set $H$. We extend this situation by adding an extra set $U$, the set of forbidden states. We suppose that $U$ is disjoint from $E$. Now, the taboo set is $H = \mathcal{X} \setminus  (U \cup E)$. 

We follow in this section the definition of safety from \cite{BuWiBou2020}. For each state in $\mathcal{X}$, the safety function gives the  probability that the realisations hit the forbidden set $U$ before reaching the target set $E$.  

We consider the problem of finding a policy $\pi$ such that the safety function
\begin{align*}
	S_{\pi}(i) := \mathds{P}^i[\tau_U < \tau_{E}]=\mathds{P}[\tau_U < \tau_{E}|~X_0 = i] \leq p,
\end{align*}
where  $\tau_A$ is the first hitting time of a set $A$. 
We have again suppressed the policy $\pi$ in the notation, $\mathds{P} = \mathds{P}(\pi)$.

To compute the safety function $S_{\pi}$, we  make again use of the evolution equation \eqref{ev_eq_function}  with the initial distribution $\mu$ concentrated at $i$, and the stopping time $\tau = \tau_{E \cup U}$, i.e., the first hitting time of the union $E \cup U$,
\begin{align*} 
    \langle \lambda_\tau^i, f \rangle = f(i)  + \langle \gamma_{<\tau}^{i}, \mathcal{L}(\pi) f \rangle.
\end{align*}
We observe that the safety function $S_{\pi}(i)=\lambda_{\tau}^{i}(U)$.
We unfold the evolution equation,
\begin{align}\label{Eq.MeasureDirichlet}
    \sum_{k \in U \cup E} \lambda_{\tau}^{i}(k) f(k)=f(i)+\sum_{j \in H} \gamma_{< \tau}^{i}(j) (\mathcal{L}(\pi) f)(j).
\end{align}
Explicitly, for the function $f$ such that $f(j) = 0$ for $j \in E$, $f(j) = 1$ for $j \in U$, and $(\mathcal{L}(\pi) f)(j) = 0$ for $j \in H$, we have
\begin{align*}
    \sum_{k \in U} \lambda_{\tau}^{i}(k) =f(i).
\end{align*}

In conclusion the safety function $S_{\pi}$ is the solution $s$ of the following problem,
\begin{subequations}\label{Dirichlet_pb}
\begin{align}
    (\mathcal{L} (\pi) s )(j) &= 0\text{, } \forall j \in H \label{Eq.Dirichlet1}\\
    s(j) &= 1\text{, }\forall j \in U \label{Eq.Dirichlet2} \\
    s(j) &= 0\text{, }\forall j \in E. \label{Eq.Dirichlet3}
\end{align}
\end{subequations}
The problem \eqref{Dirichlet_pb} is known as the 
Dirichlet problem. Its solution is unique. 
Since \eqref{Dirichlet_pb} is linear in $s$, we formulate it in terms of matrices. To this end, we suppose the state are numbered in the following order: the states in $H$ are first, then in $U$, and finally in $E$. We decompose $P := P(\pi)$ as follows
\begin{align}\label{Eq.DecomposedTransitionMtx}
    P = \begin{bmatrix}
    Q & P_H^U & P_H^E \\
    P_U^H & P_U^U & P_U^E \\
    P_E^H & P_E^U & P_E^E
    \end{bmatrix}.
\end{align}
~\\~

\begin{lem}\label{Lemma.MatrixSafety}
Suppose that the Markov decision process $(X_t)$ with a policy $\pi$ is transient on $H$. 
Let 
\begin{align}\label{Eq.RHSsafety}
    K_{\pi} := P_H^U(\pi) \mathds{1}_{U},
\end{align}
where $\mathds{1}_{U}$ is the vector of $1$s in $U$. 
Then the safety function is given by
\begin{align}\label{Eq.SafetyEquation}
    S_{\pi}|_H = G(\pi) K_{\pi},
\end{align}
and it is the solution of
\begin{align}\label{Eq.FixedPointSafety}
    S_{\pi}|_H = Q(\pi) S_{\pi}|_{H} + K_{\pi}.
\end{align}
Furthermore, the sequence $(S_{\pi}^n)$ defined by
\begin{align}\label{Eq.SafetyAsDynamicSystem}
    S_{\pi}^{n+1} = Q(\pi) S_{\pi}^n + K_{\pi}
\end{align}
for an arbitrary $S_{\pi}^0$ converges point-wise to $S_{\pi}|_H$.
\end{lem}
 ~\\~
 
\begin{proof}
We apply the transition matrix $P(\pi)$ in \eqref{Eq.DecomposedTransitionMtx} to the Dirichlet problem~\eqref{Dirichlet_pb} to get \eqref{Eq.SafetyEquation}. 
From  \eqref{Eq.SafetyEquation} and \eqref{Eq.RelationGreenGenerator}, we have \eqref{Eq.FixedPointSafety}.

We regard ~\eqref{Eq.SafetyAsDynamicSystem} as a discrete time dynamical systems, and observe that the eigenvalues of $Q(\pi)$ are in the open unit disk. Consequently, the sequence $(S_{\pi}^n)$ converges to $G(\pi) K_{\pi}$. 
\end{proof}
~\\~

We extend the safety function to act on subsets of the taboo set $H$. For $A \subset H$, we define 
\begin{align*}
S_{\pi}(A) := \max_{j \in A}  S_{\pi}(j) .    
\end{align*}
We think about the set $A$ as the set of the start-points of the process. Specifically, from Lemma~\ref{Lemma.MatrixSafety},
\begin{align*}
    S_{\pi}(A) \leq p \iff I_A G(\pi) K_{\pi} \leq p \mathds{1}_A.
\end{align*}

In the last part of this section, we ask the question of what is the safest policy, or what 
is the minimum $S^*$ of the set of safety functions $S_{\pi}$ for $\pi \in \mathcal{D}^{H}$.
Having reformulated safety into dynamic programming, we can answer this question making use of Lemma~\ref{Lemma.RLAlgorithms}, or Lemma~\ref{SupremumValueunction}. Specifically, from Lemma~\ref{Lemma.RLAlgorithms}, the function $S^{*}|_{H}$ is the coordinate-wise limit of the sequence $(S^n)$ defined by
\begin{align}\label{Eq.IterativeBellaman0}
    S^{n+1} = \min_{\pi \in \mathcal{D}^{H}} \left[ K_{\pi} + Q(\pi)S^n \right]
\end{align}
with an arbitrary initial condition $S^0 \geq 0$. 

Furthermore, form  Lemma~\ref{SupremumValueunction}, $S^*|_H$ is the supreme of the functions $S$ that satisfy the following inequality
\begin{align*}
    \Delta(\pi) S \leq K_{\pi} \hbox{ for all } \pi \in \mathcal{D}^{H}. 
\end{align*}

\section{Safe Dynamic Programming}\label{Sec.safeDynamicProgramming}
We combine safety and dynamic programming to formulate a safe dynamic programming.

For $0 \leq p \leq 1$, we strive to find the minimum $V^*(i)$ of the cost
\begin{align}\label{Eq.Optimize}
    V_{\pi}(i) := \mathds E_{\pi} \left[\sum_{t=0}^{\tau} R_t| X_0 = i\right],
\end{align}
subject to
\begin{align}\label{Eq.WithSafety}
S_{\pi}(i) \leq p    
\end{align}
for $i \in H$. 
We combine Lemma~\ref{RelationGreenAndValue} with Lemma~\ref{Lemma.MatrixSafety} in the following theorem.
~\\~

 \begin{theorem}\label{Thm.ConstraintOptimisation}
 
Let $R_{\pi} := \mathrm{diag}_1( \pi \rho)|_H$ and $K_{\pi} := P_H^U(\pi) \mathds{1}_{U}$.
Let $\lambda:E \to \mathds{R}_{\geq 0}$.

 Suppose that the sequences of the costs $(V_{\pi})$ and safety functions $(S_{\pi})$ are define by
\begin{subequations}\label{Safety_pb}
\begin{align}
    V_{\pi} & := Q(\pi) V_{\pi} + R_{\pi}\label{SubEq.Safety} \\
     S_{\pi} & := Q(\pi) S_{\pi} + K_{\pi}. \label{SubEq.Optimality} 
\end{align}
\end{subequations}

Furthermore, let $L$ be the Lagrangian,
 \begin{align}\label{Eq.Saddle}
     L(\pi, \lambda) := V_{\pi} + \lambda \circ \left( S_{\pi} - p \mathds{1}_H\right),
 \end{align}
with a vector of Lagrange multipliers $\lambda \in \mathds{R}_{\geq 0}^H$ (a multiplier for each state $i \in H$), where $\circ$ denotes the Hadamard product, and 
 \begin{align*}
     q(\lambda) : = \min_{\pi \in \mathcal{U}^H} L(\pi, \lambda). 
 \end{align*}
 Then if the problem is feasible (there is a policy $\pi$ such that $S_\pi \leq p \mathds{1}_H$), the minimum $V^*$ of the cost \eqref{Eq.Optimize} with the constraint \eqref{Eq.WithSafety} satisfies
 \begin{align*} 
     V^*|_H = q^* \equiv \sup\{q(\lambda)|~\lambda \geq 0\}.
 \end{align*}
 \end{theorem}
~\\~

 \begin{proof}
We will show that the stated optimisation has no duality gap. To this end, we will formulate an equivalent optimisation problem, which has no duality gap. Then, the statement of the theorem will follow. 
At the outset, we observe that the solution $V_\pi$ of \eqref{SubEq.Safety} is equal the cost in \eqref{Eq.Optimize}, and the solution $S_\pi$ of \eqref{SubEq.Optimality} is the safety function. 
We fix the Lagrangian multiplier $\lambda$, 
 
 \begin{align*}
     L(\pi, \lambda) = G(\pi) \left(R_{\pi} + (K_{\pi} - p \Delta(\pi)\mathds{1}_H) \circ \lambda)\right)
 \end{align*}
 Consequently,
 \begin{align*}
     \Delta(\pi) ( L(\pi, \lambda)  + p  \lambda ) = R_{\pi} + K_{\pi} \circ \lambda. 
 \end{align*}
Employing Lemma~\ref{SupremumValueunction}, we reformulate the original problem of finding the minimum of the cost \eqref{Eq.Optimize} with the constraint \eqref{Eq.WithSafety} as the following optimisation. Let
  \begin{align*}
    \mathcal{V}_{\pi} := \{ (l,\lambda) \in \mathds{R}_{\geq 0}^E \times R_{\geq 0}^E |~\Delta(\pi) l \leq R_{\pi} + K_{\pi} \circ \lambda\},
 \end{align*}
 and
 \begin{align*}
 \mathcal{V} =  \bigcap_{\pi \in \mathcal{D}^H} \mathcal{V}_{\pi}.    
 \end{align*}
Then, 
 \begin{align*}
 q(\lambda) = \sup_l (l - p \lambda)
 \end{align*}
 subject to 
 $$
 l \in \mathcal{V} \cap (\mathds{R}^H \times \{\lambda\}).
 $$ 
 Since each of the sets $\mathcal{V}_{\pi}$ is convex, and an arbitrary intersection of convex sets is convex, we conclude that $\mathcal{V}$ is convex. 
 
 We observe that
 \begin{align*}
q^* =  \sup_{\lambda} \sup_l (l - p\lambda)     
 \end{align*}
 subject to $(l,\lambda) \in \mathcal{V}$. 
 By the convexity of $\mathcal{V}$
\begin{align*}
    q^* = \sup_l \sup_{\lambda}  (l - p\lambda)  = \sup_{(l, \lambda)}  (l - p\lambda)  
\end{align*}
subject to $(l,\lambda) \in \mathcal{V}$. In conclusion, there is no duality gap, and $V^*|_H=q^*$.
 $\Box$
 \end{proof}

We have the following corollary from Theorem~\ref{Thm.ConstraintOptimisation}.
~\\~

 \begin{cor}
 The minimum $V^*$ of the cost \eqref{Eq.WithSafety} with the constraint \eqref{Eq.Optimize} satisfies
 \begin{align} \label{Eq.SaddleBellman}
     V^*|_H = \sup_{\lambda \geq 0} \min_{\pi \in \mathcal{D}^H} \left( Q(\pi) V^* +  R_{\pi} + K_{\pi} \circ \lambda - p\Delta(\pi)\lambda \right). 
 \end{align}
 \end{cor}

\begin{proof}
 To prove this statement, we observe that for a fixed policy $\pi$, 
 $S_{\pi}|_H = G(\pi) K_{\pi},$ $V_{\pi}|_H = G(\pi) R_{\pi},$ and $\Delta(\pi) G(\pi) = I_H.$
 
 We multiply the equality \eqref{Eq.Saddle} by $\Delta(\pi)$, 
 \begin{align*}
 L(\pi, \lambda) =  \left( Q(\pi) L(\pi, \lambda) +  R_{\pi} + K_{\pi} \circ \lambda - p\Delta(\pi)\lambda \right).     
 \end{align*}
 Hence, for fixed $\lambda$, by Lemma~\ref{Lemma.RLAlgorithms}, the sequence $(q^n(\lambda))$ defined by
  \begin{align*}
 q^{n+1}(\lambda)  =  \min_{\pi \in \mathcal{D}^H} ( Q(\pi) q^n(\lambda) +  R_{\pi} + K_{\pi} \circ \lambda - p\Delta(\pi)\lambda ).     
 \end{align*}
 converges to $q(\lambda)$ for any initial value $q^0(\lambda)$, and
  \begin{align*}
 q(\lambda) =  \min_{\pi \in \mathcal{D}^H} \left( Q(\pi) q(\lambda) +  R_{\pi} + K_{\pi} \circ \lambda - p\Delta(\pi)\lambda \right).     
 \end{align*}
\end{proof}

\section{Algorithms for Safe Dynamic Programming}\label{Section.SafeDynamicProgramming}
 
 We propose two solutions to the safe dynamic programming. The first uses the construction in the proof of Theorem~\ref{Thm.ConstraintOptimisation}. Specifically, for pure policies, i.e., $\pi \in H \times \mathcal{U}$, the proposed optimisation boils down to a linear programming. In the second algorithm, the set of policies is beforehand restricted to those which are safe.

\subsection{Linear Programming for Pure Policies}\label{Subsection.LP}
The proof of Theorem~\ref{Thm.ConstraintOptimisation} was carried out indirectly by reformulating the original problem to an equivalent optimisation and showing that the new optimisation has no duality gap. This new optimisation is motivated by \cite{Altman99constrainedmarkov}, and in the current situation of safe dynamic programming with stopping time-optimisation pronounces
\begin{align*}
    V^* = \sup_{(l,\lambda) \in \mathcal{V}} (l - p\lambda), 
\end{align*}
where 
  \begin{align*}
    \mathcal{V} :=&  \{ (l,\lambda) \in \mathds{R}_{\geq 0}^H \times R_{\geq 0}^H | \\
    & \Delta(\pi) l \leq R_{\pi} + K_{\pi} \circ \lambda \hbox{ for all } \pi \in \mathcal{D}^H\}.
 \end{align*}
For pure policies $\pi \in \mathcal{U}^H$, the above optimisation becomes a linear programming problem. 

 \subsection{Policy Selection}
Also in this subsection, we suppose that the policies are pure. For a $p\in [0, 1]$, we define the set of admissible policies by
 \begin{align*}
     \Pi_p := \{\pi \in H \times \mathcal{U}|~S_{\pi} \leq p\}.
 \end{align*}
 In other words, $\pi \in \Pi_p$ if and only if  there is $x \in \mathds{R}_{\geq 0}^H$ such that  $K_{\pi} = \Delta(\pi)(p - x)$. In the next lemma, we provide another characterisation of the set of admissible policies. 
~\\~

\begin{lem}
Let $f^i_\pi = \Delta(\pi) e^i$, where $e^i$ are the standard basis of $\mathds{R}^H$. Let $M_{\pi} := p \Delta(\pi) \mathds{1}_H - K_{\pi}$. 
The set $\Pi_p$ of admissible policies is
\begin{align*}
    \Pi_p := \{ \pi \in H \times \mathcal{U} |~ M_{\pi} = \sum_{i \in H} \alpha^i f^i_\pi, \hbox{ for some }\alpha^i \geq 0\}.
\end{align*}
\end{lem}
~\\~
 
 \begin{proof}
The system is $p$-safe if only if $p \mathds{1}_H - G(\pi) K_{\pi} \geq 0$, or equivalently
\begin{align*}
    p \mathds{1}_H - G(\pi) K_{\pi} \in  \mathrm{conv} \{e^i_\pi |~i \in H\}.
\end{align*}

We re-write $p$ safety as
 \begin{align*}
M_{\pi} \in \mathrm{conv} \{f^i_\pi |~i \in H\}.
 \end{align*}
 \end{proof}
~\\~

 Consequently, by Lemma~\ref{Lemma.RLAlgorithms}, the value function $V^*|_H$ is the solution of
\begin{equation} \label{eq.BellmanEquationcompact0}
    V = \min_{\pi \in  \Pi_p} \left[ R_{\pi} + Q(\pi)V \right].
\end{equation}
Furthermore, the function $V^{*}|_{H}$ is the (coordinate-wise) limit of the sequence $(V^n)$ defined by
\begin{align}\label{Eq.IterativeBellaman}
    V^{n+1} = \min_{\pi \in \Pi_p} \left[ R_{\pi} + Q(\pi)V^n \right].
\end{align}

Alternatively,  the value function $V^*$
satisfies
\begin{align}\label{Eq.ValueSupremumFormula}
    V^*|_H= \sup \mathcal{V},
\end{align}
where
\begin{align*}
    \mathcal{V} := \{V \in \mathds{R}^H |~\Delta(\pi) V \leq R_{\pi} \hbox{ for all } \pi \in \Pi_p\}.
\end{align*}

\section{Locality and Relative Safety} \label{Sec.LocalSafeDynamicProg}

 In the remaining part of this article, we briefly discuss the locality of dynamic programming. Subsequently, we define an alternative concept of safety, for which we propose safe dynamic programming.
 
 We say that a decision algorithm is local if the policy's choice at the state $i$ leans only upon the local information. Specifically, the Bellman's equation in \eqref{eq.BellmanEquationcompact} 
 writes
 \begin{align*}
 V(i) = \min_{d \in \mathcal{D}} \left[ \sum_{u \in \mathcal{U}} \left(\rho(u,i)  + \sum_{j \in H}p_{ij}(u)  V(j) \right) d(u) \right].    
 \end{align*}
 In other words, to compute the distribution of the control at the state $i$, only information ``exchange'' from the state $i$ to its neighbours is necessary. Here, a neighbour is any state $j$ where the transition probability $p_{ij}$ is nonzero.
 
 Similarly, the algorithm for computing the safety function, with a fixed policy $\pi =\{ d(i) \in \mathcal{D} \}_{i \in H}$,
 \begin{align*}
     S_{\pi}(i) = \sum_{u \in \mathcal{U}} \left(  \sum_{k \in U} p_{ik}(u) + \sum_{j \in H}p_{ij}(u)  S_{\pi}(j) \right) d(u)
 \end{align*}
 is local. The above local formulations are in contrast with the safe dynamic programming algorithms proposed in Section~\ref{Section.SafeDynamicProgramming}. Those ``global'' approaches might be prohibitive for large state spaces. Therefore, we propose a local approach to safe dynamic programming leaning upon a modified concept of safety. 

In definition~\eqref{Def.RelativeSafety} below, for a non-negative real number $q$, we define $q$-relative safety. Shortly, $q$ is the ratio of the probability of hitting the forbidden set $U$ to the probability of reaching the target set $E$. 
~\\~

\begin{defn}\label{Def.RelativeSafety}
Let $K_{\pi} := P_H^U(\pi) \mathds{1}_{U}$ and $L_{\pi} = P_H^E(\pi) \mathds{1}_{E}$, where $P_H^U(\pi)$ and $P_H^E(\pi)$ are the transition probability matrices in \eqref{Eq.DecomposedTransitionMtx}.

Let $q \in \mathds{R}_{\geq 0}$. The Markov decision process with a policy $\pi$ is $q$-relatively safe if and only if
\begin{align}\label{Eq.relativeSafety}
    K_{\pi} \leq q L_{\pi}. 
\end{align}
\end{defn}
~\\~

We observe that the notion of $q$-relative safety is local since the evaluation of whether the process is $q$-relatively safe at a state $i$ can be determined with the help of information available from the neighbour states.

The probability of hitting first the forbidden set $U$ is the safety function $S_\pi$ defined in Section~\ref{Section.Safety},
\begin{align*}
S_{\pi} = G(\pi) K_{\pi}.
\end{align*}
By exchanging the roles of the sets $U$ and $E$, the probability of reaching the target set $E$ first is given by
\begin{align*}
T_{\pi} = G(\pi) L_{\pi}.     
\end{align*}
Since $G(\pi)$ is a positive (possibly infinite dimensional) matrix, we conclude that if the process is $q$-relatively safe then
\begin{align*}
    S_{\pi} \leq q T_{\pi}. 
\end{align*}
In the next lemma, we relate $q$-relative safety with $p$-safety. 
~\\~

\begin{lem}
A Markov chain is $p$-safe only if it is $q$-relatively safe with $q = \frac{p}{1-p}$.
\end{lem}
~\\~

\begin{proof}
We compress the notation by suppressing the explicit reference to the policy. We observe that
\begin{align*}
    \mathds{1}_H = P_H^H \mathds{1}_H + P_H^U \mathds{1}_U + P_H^E \mathds{1}_E.
\end{align*}
Hence,
\begin{align*}
    \Delta \mathds{1}_H  = K_{\pi} + L_{\pi},
\end{align*}
and
\begin{align*}
    \mathds{1}_H  = S + T.
\end{align*}
Consequently,
\begin{align*}
    S-p\mathds{1}_H = (1-p) S - p T,
\end{align*}
and
\begin{align*}
   S \leq p \mathds{1}_H ~\hbox{ if and only if }~ S - \frac{p}{1-p} T \leq 0. 
\end{align*}
On the other hand, $q$-relative safety implies $ S_{\pi} \leq q T_{\pi}.$~\\~
\end{proof}
~\\~

We conclude this section by re-formulating the safe dynamic programming using the newly established concept of $q$-relative safety. For $q \geq 0$, we compute 
\begin{align*}
   V^*(i) = \min_{\pi} \mathds{E}_{\pi}\left[ \sum_{t=0}^{\tau} R_t |~X_0 = i \right],
\end{align*}
where the policies $\pi$ are $q$-relatively safe. 
Using the relation~\eqref{Eq.relativeSafety}, we define the set of admissible policies with respect to $q$-relative safety,
 \begin{align*}
     \overline{\Pi}_q := \{\pi \in \mathcal{D}^H|~K_{\pi} \leq q L_{\pi}\}.
 \end{align*}
As a consequence, $V^*|_H$ is the solution of the Bellman's equation
\begin{equation*} 
    V = \min_{\pi \in \overline{\Pi}_q} \left[ R_{\pi} + Q(\pi)V \right].
\end{equation*}

Furthermore, the function $V^{*}|_{H}$ is the limit of the sequence $(V^n)$ defined by
\begin{align*}
    V^{n+1} = \min_{\pi \in \overline{\Pi}_q} \left[ R_{\pi} + Q(\pi)V^n \right]
\end{align*}
with an arbitrary initial condition $V^0 \geq 0$. 

\section{Illustrative Example}\label{Sec.Example}
To illustrate the usage of the central concepts developed in this paper, we consider the following simple example. Suppose that the state space is $\mathcal X = \{a, b, c, d, e\}$, where $d$ is the forbidden state, $e$ is the goal. Consequently, the taboo set is $H = \{a, b, c \}.$ We also suppose that the only nonzero transition probabilities are $p_{ba}(u)$, $p_{bc}(u)$, $p_{ca}(u)$, $p_{ad}(u)$, $p_{ae}(u)$ for $u \in \{u_1, u_2\}$, and $p_{bc}(u_1) > p_{bc}(u_2)$, $p_{ae}(u_1) \geq 1/3$. Specifically, $p_{ca}(u) = 1$, $u \in \{u_1, u_2\}$. As the consequence, the transition matrix $Q$ is
\begin{align*}
    Q(\pi) = \begin{bmatrix}
    0 & 0 & 0 \\ p_{ba}(\pi) & 0 & p_{bc}(\pi) \\ 1 & 0 & 0
    \end{bmatrix},
\end{align*}
where the transition probabilities in the entries of $Q$ are
\begin{align*}
    p_{ba}(\pi) &= p_{ba}(u_1) \pi_{b u_1} + p_{ba}(u_2) \pi_{b u_2}, \\
     p_{bc}(\pi) &= p_{bc}(u_1) \pi_{b u_1} + p_{bc}(u_2) \pi_{b u_2}.
\end{align*}
Without mentioning, we use the observation that $p_{ba}(\pi) + p_{bc}(\pi) = 1$.

Let $q=2$. Then the $q$-relative safety implies
\begin{align*}
    p_{ad}(\pi) \leq 2 p_{ae} (\pi).
\end{align*}
After unfolding the transition probabilities,
\begin{align*}
    p_{ad}(u_1)\pi_{a u_1} + p_{ad}(u_2)\pi_{a u_2} \leq 2( p_{ae}(u_1)\pi_{a u_1} + p_{ae}(u_2)\pi_{a u_2} ), 
\end{align*}
and hence, 
\begin{align*}
   (p_{ae}(u_1) + p_{ae}(u_2)) \pi_{a u_1} \geq p_{ae}(u_2) + 1/3. 
\end{align*}
Specifically, for $\pi_{a u_1} = 1$, the process is $q$-relatively safe. 

We compute the safety function,
\begin{align*}
    S_{\pi}|_{H} = G(\pi) K_{\pi} = \begin{bmatrix}
    1 & 0 & 0 \\ 1 & 1 & p_{bc}(\pi) \\ 1 & 0 & 1
    \end{bmatrix}
    \begin{bmatrix}
    p_{ad}(\pi) \\ 0 \\ 0
    \end{bmatrix} = p_{ad}(\pi) \begin{bmatrix}
    1 \\ 1 \\ 1
    \end{bmatrix}.
\end{align*}
Specifically, for $\pi_{a u_1} = 1$, the process is $p = 2/3$ safe.

We suppose that the reward function $\rho(a,u) = 1$, $\rho(b,u) = 2$, $\rho(c,u) = 3$, and $\rho(d,u) = \rho(e,u) = 0$ for all actions $u$ in  $\{u_1, u_2\}$. Hence $R := R_{\pi} = \begin{bmatrix}1 & 2 & 3 \end{bmatrix}^\mathrm{T}$. We remark that $Q(\pi)^3 = 0$, and
\begin{align*}
    G(\pi)R &= Q(\pi) Q(\pi) R + Q(\pi) R + R \\
    &= \begin{bmatrix}
    1 & 2 + p_{ba}(\pi) + 4 p_{bc}(\pi)  & 4 
    \end{bmatrix}^{\mathrm{T}} \\
    &= \begin{bmatrix}
    1 & 3 + 3 p_{bc}(\pi)  & 4 
    \end{bmatrix}^{\mathrm{T}}
    .
\end{align*}
The minimizing policy is  $\pi_{b u_2} = 1$, and 
$$\min_{\pi}G(\pi)R = \begin{bmatrix}  1 & 3+3p_{bc}(u_2)  & 4 \end{bmatrix}^{\mathrm{T}}.$$

Lastly, we use the iterative procedure for finding the solution to Bellman's equation. We start with $V^0 = 0$. Consequently, $V^1 = \begin{bmatrix}  1 & 2  & 3 \end{bmatrix}^{\mathrm{T}}$, and
\begin{align*}
V^2 &= \min_\pi
\begin{bmatrix}  1 & 3+2p_{bc}(\pi)  & 4 \end{bmatrix}^{\mathrm{T}} = \begin{bmatrix}  1 & 3+2p_{bc}(u_2)  & 4 \end{bmatrix}^{\mathrm{T}}, \\
V^3 &= \min_\pi
\begin{bmatrix}  1 & 3+3p_{bc}(\pi)  & 1 \end{bmatrix}^{\mathrm{T}} = \begin{bmatrix}  1 & 3+3p_{bc}(u_2)  & 4 \end{bmatrix}^{\mathrm{T}}, \\
V^4 & =  \begin{bmatrix}  1 & 3+3p_{bc}(u_2)  & 4 \end{bmatrix}^{\mathrm{T}}.
\end{align*}

In conclusion, the example shows how to compute $p$-safety, $q$-relative safety, and minimizing policy. We illustrate that in all three computations, the Green operator plays a crucial role.

\section{Conclusion}
In this work, we have reformulated dynamic programming for Markov decision processes in terms of the evolution equation and the Green kernel. We have used this formulation to devise dynamic programming for computing safety. The main contribution is establishing a setting for safe dynamic programming - a dynamic program that adheres to safety specifications.

\appendices

\bibliographystyle{IEEEtran} 
\bibliography{autosam}


\end{document}